\documentclass[12pt, a4paper, leqno]{amsart}

\usepackage{amsmath}
\usepackage{amsfonts}
\usepackage{amssymb}
\usepackage{times}
\usepackage{comment,graphicx}
\usepackage[T1]{fontenc} 
\usepackage{url}
\usepackage{color, esint}

\newcommand{\norm}[1]{\ensuremath{\Vert #1 \Vert}}
\newcommand{\abs}[1]{\ensuremath{\vert #1 \vert}}

\DeclareMathOperator*{\essinf}{ess\,inf}

\newcommand{\Z}{\mathbb{Z}}
\newcommand{\R}{\mathbb{R}}
\newcommand{\Rn}{{\mathbb{R}^n}}

\newcommand{\W}{\mathcal{W}}
\newcommand{\V}{\mathcal{V}}
\newcommand{\M}{\mathcal{M}}

\newcommand{\I}{\mathcal{I}}

\newcommand{\p}{{p(\cdot)}}
\newcommand{\q}{{q(\cdot)}}
\newcommand{\px}{{p(\cdot)}}
\newcommand{\rx}{{r(\cdot)}}
\newcommand{\qx}{{q(\cdot)}}
\newcommand{\ax}{{\alpha(\cdot)}}

\newcommand{\PP}{{\mathcal P_1}}
\newcommand{\PPz}{{\mathcal P_0}}
\newcommand{\PPln}{{\mathcal P_1^{\log}}}
\newcommand{\PPzln}{{\mathcal P_0^{\log}}}

\newcommand{\dive}{\ensuremath{\mathrm{div}}}

\def\Xint#1{\mathchoice
     {\XXint\displaystyle\textstyle{#1}}%
        {\XXint\textstyle\scriptstyle{#1}}%
       {\XXint\scriptstyle\scriptscriptstyle{#1}}%
          {\XXint\scriptscriptstyle\scriptscriptstyle{#1}}%
             \!\int}
           \def\XXint#1#2#3{{\setbox0=\hbox{$#1{#2#3}{\int}$}
                \vcenter{\hbox{$#2#3$}}\kern-.5\wd0}}
            
\def\vint{\Xint-}

\def\phi{\varphi}
\def\rho{\varrho}
\def\epsilon{\varepsilon}
\def\abs#1{\mathopen| #1 \mathclose|}
\def\norm#1{\mathopen\| #1 \mathclose\|}

\def\le{\leqslant}
\def\leq{\leqslant}
\def\ge{\geqslant}
\def\geq{\geqslant}
\def\phi{\varphi}
\def\rho{\varrho}

\def\esssup{\operatornamewithlimits{ess\,sup}}

\def\min{\qopname\relax o{min}}
\def\px{{p(\cdot)}}
\def\loc{{\rm loc}}

\def\wL{{w\text{-}L}}
\newcommand{\weakL}{{w\text{-}L}}
\newcommand{\weakLpx}{{w\text{-}L^\px}}

\newcommand{\ud}[0]{\ensuremath{\,\textrm{d}}}

\theoremstyle{plain}
\newtheorem{theorem}[equation]{Theorem}
\newtheorem{lemma}[equation]{Lemma}
\newtheorem{proposition}[equation]{Proposition}
\newtheorem{corollary}[equation]{Corollary}

\theoremstyle{definition}
\newtheorem{definition}[equation]{Definition}

\newtheorem{example}[equation]{Example}

\theoremstyle{remark}

\numberwithin{equation}{section}

\def\be{\begin{equation}}
\def\ee{\end{equation}}

\begin{document}

\title{Riesz and Wolff potentials and elliptic equations in variable exponent weak Lebesgue spaces}



\title[Riesz and Wolff potentials and elliptic equations]{Riesz and Wolff potentials and elliptic equations in variable exponent weak Lebesgue spaces}

\author{A.\ Almeida, P.\ Harjulehto, P.\ H\"ast\"o
  and T.\ Lukkari} \date{\today}


\maketitle

\begin{abstract}
We prove optimal integrability results for solutions of the $\px$-Laplace equation in the scale of (weak) Lebesgue spaces.
To obtain this, we show that variable exponent Riesz and Wolff potentials map $L^1$ to variable exponent weak Lebesgue spaces.
\end{abstract}
  

\section{Introduction}

In this paper we study the mapping properties of variable exponent Riesz and Wolff
potentials on weak $L^{\p}$ spaces, denoted by $\wL^{\p}$. Our
interest stems mainly from the following problem, whose solution is presented in Section~\ref{sect:pde}.  Consider
appropriately defined weak solutions to the boundary value problem
\begin{equation}\label{eq:intro-px-equation}
  \begin{cases}
    -\dive(\abs{\nabla u}^{p(x)-2}\nabla u)=f&\text{in }\Omega,\\
    u=0& \text{on }\partial \Omega
  \end{cases}
\end{equation}
when the data $f$ is merely an $L^1$ function. We refer to \cite{HarHLN10} for an extensive survey of such equations with non-standard growth.
Based on the constant exponent case and computations
on explicit solutions, one expects in the $L^1$-situation that
\begin{equation}\label{eq:intro-incl}
  u\in \wL_\loc^{\frac{n}{n-\p}(\p-1)}(\Omega)\quad \text{and}\quad
  \abs{\nabla u }\in \wL_\loc^{\frac{n}{n-1}(\p-1)}(\Omega).
\end{equation}
By earlier results of Sanch\'on and Urbano
\cite[Remark~3.3]{SU}, the gradient belongs to the space $\wL_{\loc}^{\frac{n
    (\p-1)}{n-1} - \epsilon}(\Omega)$, while B\"ogelein and Habermann
\cite{BogH10} proved that it is in $L_{\loc}^{\frac{n
    (\p-1)}{n-1} - \epsilon}(\Omega)$, for any $\epsilon >0$. By
elementary properties of weak spaces
(Proposition~\ref{prop:weak-strong}) these two results are in fact equivalent.
However, as \eqref{eq:intro-incl} is the borderline case $\epsilon=0$,
it has turned out to be hard to reach. As in the constant exponent
case, when $\epsilon=0$ the inclusions into the (strong) Lebesgue space do not hold.

Our approach to this problem relies on the recent pointwise potential
estimates for solutions and their gradients to problems with $L^1$ or
measure data, see \cite{DM2,DM1,Min}. The case of equations similar
to \eqref{eq:intro-px-equation} is covered in \cite{BogH10}. The
potential that appears in the nonlinear situation is the Wolff
potential, given by
\[
  \W^{f}_{\alpha,p}(x):=\int_0^\infty\left( \frac{\int_{B(x,r)} \abs{f(y)} \, dy }{r^{n-\alpha p}}
  \right)^{1/(p-1)}\frac{dr}{r}.
\]
At a given point $x$, a solution to \eqref{eq:intro-px-equation} is
controlled by $\W^{f}_{1,p(x)}(x)$, and its gradient is controlled by
$\W^{f}_{1/p(x),p(x)}(x)$.

These estimates are the nonlinear counterparts of representation
formulas, as properties of solutions may be deduced from the
properties of the potentials. Our aim is to exploit this, and
establish a local version of \eqref{eq:intro-incl} by proving that the
Wolff potential $\W^{f}_{\alpha(x),p(x)}(x)$ has the appropriate
mapping properties.  This answers the open problem posed by Sanch\'on
and Urbano \cite[Remark~3.3]{SU} and completes the generalization of
the Wolff-potential approach for \eqref{eq:intro-px-equation} started
by B\"ogelein and Habermann in \cite{BogH10}.

The usual way to look at the mapping properties of the Wolff potential
is to estimate it pointwise by the Havin-Maz'ya potential (see
\cite{HM}), which is an iterated Riesz potential.  Thus we study the
mapping properties of the Riesz potential as well. For (strong)
Lebesgue spaces these properties are well known, see \cite{Die04b,SamSV07,Sam98} and \cite[Section~6.1]{DieHHR11}. Here we deal with the
novel case of weak Lebesgue spaces.

Our first result, Theorem~\ref{thm:R-L1-to-weak}, is the strong-to-weak estimate for the Riesz
potential $\I_\ax$.  We show that
\[
\I_\ax : L^\rx(\Omega) \to \wL^{r^\#_\alpha(\cdot)}(\Omega),
\]
where $\Omega$ is an open, bounded set in $\R^n$, the target space is a weak
variable exponent Lebesgue space and $r^\#_\alpha:= nr/(n-\alpha r)$
is the (pointwise) Sobolev conjugate of $r$. For $r^-:= \inf r >1$,
strong-to-strong boundedness has been known for ten years
\cite{Die04b}, so the novelty lies in the inclusion of the case
$r^-=1$. In contrast to the constant exponent case, this is not enough
for us; surprisingly, the fact that
\[
f\in L^\rx(\Omega) \quad\Longrightarrow  \quad (\I_\ax f)^\qx \in
\wL^{\frac{r^\#_\alpha(\cdot)}{\qx}}(\Omega)
\]
for every $\log$-H\"older continuous positive function $q$ requires a
separate proof. This proof is based on pointwise estimate between the
Riesz potential and the Hardy--Littlewood maximal operator.

Then we study how the Riesz potential acts on weak Lebesgue spaces, as
this situation will inevitably happen when dealing with the Wolff
potential on $L^1$. This turns out to be a difficult question because
the weak Lebesgue spaces are not well-behaved. We show that the weak
Lebesgue space is an interpolation space
(Theorem~\ref{thm:LorentzInfinity}).  This allows us to use real
interpolation to get weak-to-weak boundedness of the maximal operator:

\begin{theorem}\label{thm:strongWeakM}
  Let $p$ be a bounded measurable function with $p^-> 1$. If $\M:
  L^\px (\Rn) \to L^\px(\Rn)$ is bounded, then so is $\M:
  \weakL^{\px}(\Rn) \to \weakL^{\px}(\Rn) $.

  In particular, $\M: \weakL^{\px}(\Rn) \to \weakL^{\px}(\Rn) $ is
  bounded when $p$ is $\log$-H\"older continuous and $p^->1$.
\end{theorem}

With a complicated application of Hedberg's trick, we then
prove in Theorem~\ref{thm:I-is-bdd-weak} that
\[
f\in \weakL^{q(\cdot)}(\Omega) \text{ and }
|f|^{q(\cdot)/q^\#_\alpha(\cdot)}\in
\weakL^{q^\#_\alpha(\cdot)}(\Omega) \quad\Longrightarrow  \quad
\I_\ax f \in \weakL^{q^\#_\alpha(\cdot)}(\Omega).
\]
We combine these results, and obtain in Theorem~\ref{thm:wolff} that
\begin{equation}\label{eq:wolff-map-intro}
f\in L^1 (\Omega)
\quad\Longrightarrow \quad
\W^f_{\alpha(x), p(x)} \in \wL^{\frac{n(\p-1)}{n-\ax\px }}(\Omega).
\end{equation}

A combination of \eqref{eq:wolff-map-intro} and the pointwise
potential estimates now yields \eqref{eq:intro-incl}, provided that an
appropriate notion of solutions to \eqref{eq:intro-px-equation} is
used.  This requires some care, as $L^1(\Omega)$ is not contained in the
dual of the natural Sobolev space $W^{1,\p}_0(\Omega)$.  Here we use
the notion of \emph{solutions obtained as limits of approximations},
or SOLAs for short. The idea is to approximate $f$ with more regular
functions, prove uniform a priori estimates in a larger Sobolev space
$W^{1,\q}_0(\Omega)$, and then pass to the limit by compactness
arguments. This way, one finds a function $u\in W^{1,\q}_0(\Omega)$ such
that \eqref{eq:intro-px-equation} holds in the sense of distributions.
See e.g. \cite{BG,BG2,KM} for a few implementations of this basic
idea, and \cite{L2,SU} for equations similar to the $\p$-Laplacian.
In fact, the same approximation approach is used in proving the
potential estimates.

A representative special case of what comes out by combining nonlinear
potential estimates and our results about the Wolff potential is the
following theorem.
\begin{theorem}\label{thm:SOLA-weak0}
  Let $f \in L^1(\Omega)$, and let $p$ be bounded and H\"older
  continuous with $p^-\geq 2$.  Suppose that $u$ is a SOLA to
  \eqref{eq:intro-px-equation}. Then
  \begin{displaymath}
    u\in \wL_{\loc}^{\frac{n(\p-1)}{n-\p}}(\Omega) \quad\text{and}\quad \abs{\nabla u}\in \wL_{\loc}^{\frac{n (\p-1)}{n-1}}(\Omega).
  \end{displaymath}
\end{theorem}
In other words, \eqref{eq:intro-incl} holds locally under suitable
assumptions.  Similar results also follow for the fundamental objects
of nonlinear potential theory, the $\p$-superharmonic
functions. Finally, by examining the counterpart of the fundamental
solution (Example~\ref{eg:fundamental-sol}) we show that the
exponents in Theorem \ref{thm:SOLA-weak0} are sharp, as expected.


\section{Notation}

We write simply $A\lesssim B$ if there is a constant $c$ such that $A\le cB$.  We also use the notation $A\approx B$ when $A\lesssim B$ and $A\gtrsim B$. For compatible vector spaces, the space $X\cap Y$ is defined by the norm $\|f\|:= \max\{\|f\|_X,\|f\|_Y\}$ while $X+Y$ is defined
by $\|f\|:= \inf_{f_1+f_2=f} \|f_1\|_X + \|f_2\|_Y$.

Let $U\subset \Rn$. For $g:U \to \R$ and $A\subset U$ we denote
\[
g^+_A := \esssup_{x\in A} g(x)
\quad\text{and}\quad
g^-_A := \essinf_{x\in A} g(x)
\]
and abbreviate $g^+ := g^+_U$ and   $g^- := g^-_U$.
We say that $g:U \to \R$ satisfies
the \textit{local $\log$-H\"{o}lder continuity} condition if
\[
|g(x)-g(y)|\leq \frac{c}{\log(e+1/|x-y|)}
\]
for all $x,y\in U$.  We will often use the fact that $g$ is
locally $\log$-H\"{o}lder continuous if and only if
\begin{equation}\label{eq:logH-osc}
  |B|^{g^-_B - g^+_B}\lesssim 1
\end{equation}
for all balls $B \cap U \neq \emptyset$.
If
\[
|g(x)-g_\infty|\leq \frac{c'}{\log(e+|x|)}
\]
for some $g_\infty\ge 1$, $c'>0$ and all $x\in U$,
then we say $g$ satisfies the
\textit{$\log$-H\"{o}lder decay condition} (\textit{at infinity}).
If both conditions are satisfied, we simply speak of
\textit{$\log$-H\"{o}lder continuity}. By the \textit{$\log$-H\"{o}lder constant} we mean $\max\{c, c'\}$.

By a \textit{variable exponent} we mean a measurable function $p:U \to (0,
\infty)$ such that $0< p^-\le p^+ <\infty$.
The set of variable exponents is denoted by $\PPz(U)$;
$\PP(U)$ is the subclass with $1\le p^-$.
By $\PPzln(U)$ and $\PPln(U)$ we denote the respective subsets consisting of
$\log$-H\"{o}lder continuous exponents.

We define a \emph{modular} on the set of measurable functions by
setting
\[
\varrho_{L^{p(\cdot)}(U)}(f) :=\int_{U} |f(x)|^{p(x)}\,dx.
\]
The \emph{variable exponent Lebesgue space $L^{p(\cdot)}(U)$}
consists of all the measurable functions $f\colon U\to \R$ for which the
modular $\varrho_{L^{p(\cdot)}(U)}(f)$ is finite. The Luxemburg norm on this space is defined as
\[
\|f\|_{L^{p(\cdot)}(U)}:=
\inf\Big\{\lambda > 0\,\colon\, \varrho_{L^{p(\cdot)}(U)}\big(\tfrac{f}\lambda\big)\leq 1\Big\}.
\]
Equipped with this norm, $L^{p(\cdot)}(U)$ is a Banach space.
We use the abbreviation $\|f\|_\px$ to denote the norm in the whole space under
consideration. The norm and the modular are related by the inequalities
\begin{equation}
  \label{eq:mod-norm}
  \min\{\|f\|_{L^{p(\cdot)}(U)}^{p^+},\|f\|_{L^{p(\cdot)}(U)}^{p^-}\}\leq
  \varrho_{L^{p(\cdot)}(U)}(f)
  \leq \max\{\|f\|_{L^{p(\cdot)}(U)}^{p^+},\|f\|_{L^{p(\cdot)}(U)}^{p^-}\}.
\end{equation}

For open sets $U$, the \emph{variable exponent Sobolev space
  $W^{1,p(\cdot)}(U)$} consists of functions $u\in L^{p(\cdot)}(U)$
whose distributional gradient $\nabla u$ belongs to
$L^{p(\cdot)}(U)$. The norm
\[
\|u\|_{W^{1,p(\cdot)}(U)}:=\|u\|_{L^{p(\cdot)}(U)}+\|\nabla u\|_{L^{p(\cdot)}(U)}
\]
makes $W^{1,p(\cdot)}(U)$ a Banach space.  The Sobolev space with zero
boundary values, denoted by $W^{1,p(\cdot)}_0(U)$, is the completion of
$C_0^\infty(U)$ with respect to the norm of $W^{1,p(\cdot)}(U)$. This
definition does not cause any difficulties: the assumptions on $p$ in
Section \ref{sect:pde}, where we use Sobolev spaces, are enough
to guarantee that smooth functions are dense in the Sobolev space.

More information and proofs for the above facts can be found for
example from \cite[Chapters~2,~4,~8, and~9]{DieHHR11}.

\emph{By $\Omega$ we always denote an open bounded set in $\Rn$.}

In auxiliary results we use the \textit{convention} that constants 
(implicit or explicit) depend on the 
assumptions stated in the result. For instance, in 
Proposition~\ref{prop:weak-strong} the assumptions are that 
$p,q\in \PPz(\Omega)$ and $(p-q)^->0$, 
so in this case, the implicit 
constant (potentially) depends on $p^-$, $p^+$, $q^-$, $q^+$, 
$(p-q)^-$, and on the dimension $n$.


\section{Basic properties of weak Lebesgue spaces}\label{sec:basic}

\begin{definition}\label{def:weakL}
Let $A\subset \Rn$ be measurable. A measurable function $f:A\to \R$
belongs to the \textit{weak Lebesgue space} $\wL^\p (A)$ if
\[
\| f\|_{\weakLpx(A)} := \sup_{\lambda>0} \lambda\, \| \chi_{\{|f|>
\lambda\}}\|_{L^\px(A)}< \infty.
\]
\end{definition}

The inequalities \eqref{eq:mod-norm} imply that the requirement in
Definition~\ref{def:weakL} is equivalent with
\begin{equation}
  \label{eq:weak-lp-modular}
  \sup_{\lambda >0}\int_{\{\vert f\vert>\lambda\}} \lambda^{p(x)}\,dx<\infty.
\end{equation}
Another immediate consequence of \eqref{eq:mod-norm} which we will use
in the proofs below is that
\begin{equation}
  \label{eq:weak-lp-modular2}
  \| f\|_{\weakLpx(A)}\leq 1 \quad \text{if and only if}\quad
  \sup_{\lambda >0}\int_{\{\vert f\vert>\lambda\}} \lambda^{p(x)}\,dx \leq 1.
\end{equation}

We immediately obtain the following two inclusions:
\begin{itemize}
\item
$L^\p(\Rn)\subset \wL^\p(\Rn)$
, since $\lambda \chi_{\{|f|>\lambda\}} \le |f|$;
\item
for bounded sets,
$\wL^\p(\Omega) \subset \wL^{\q}(\Omega)$ when $p\ge q$, since the inequality
$\|\cdot\|_\px \gtrsim \|\cdot\|_\qx$ holds for the corresponding strong spaces.
\end{itemize}

The following result is from \cite[Proposition~2.5]{SU}. We
present a simpler proof here.

\begin{proposition}\label{prop:weak-strong}
Let $p,q\in \PPz(\Omega)$.
If $(p-q)^->0$, then $\wL^\p(\Omega) \subset L^{\q}(\Omega)$.
\end{proposition}

\begin{proof}
Let $f \in \wL^\p(\Omega)$. We write $E_i:= \{2^{i} \le |f| < 2^{i+1}\}$  for every $i=0,1,2,\ldots$\@
Then $\Omega = \bigcup_{i=0}^\infty E_i \cup \{|f| < 1\}$. We obtain
\begin{align*}
\int_\Omega |f|^{q(x)} \, dx &\le \sum_{i=0}^\infty \int_{E_i} 2^{(i+1)q(x)} \, dx + |\Omega|\\
&\le 2^{q^+} \sum_{i=0}^\infty \int_{\{|f| \ge 2^i\}} 2^{ip(x)} 2^{-i(p(x)-q(x))} \, dx + |\Omega| \\
&\le 2^{q^+} \sum_{i=0}^\infty 2^{-i(p-q)^-} \int_{\{|f| \ge 2^i\}} 2^{ip(x)}  \, dx + |\Omega| \\
&\le 2^{q^+} \max\Big\{\|f\|_{\wL^\p (\Omega)}^{p^+},\|f\|_{\wL^\p (\Omega)}^{p^-} \Big\}  \sum_{i=0}^\infty  2^{-i(p-q)^-}+ |\Omega|  <\infty.~
\end{align*}
\end{proof}

Note that Proposition~\ref{prop:weak-strong} works not only for bounded sets but also for every open set with a finite measure.  It can be similarly proved that
\[
\wL^\p(\Rn) \subset L^{\q}(\Rn) + L^{r(\cdot)}(\Rn)
\]
for all exponents $p,q,r$ with $(p-q)^->0$ and $r\ge p$.

It is easy to show that $f\in L^\px(\Rn)$ if and only if $|f|^\qx\in L^\frac\px\qx(\Rn)$.
However, the same is not true for the weak Lebesgue space. Indeed, in this case the following property holds:

\begin{proposition}\label{prop:notWorking}
Let $p\in \PPz(\Rn)$.
Then $|f|^\qx\in \weakL^\frac\px\qx (\Rn)$ for every function
$q:\Rn\to (0,\infty)$ if and only if $f\in L^\px(\Rn)$.

If $q$ is constant, then $f\in \weakL^\px(\Rn)$ if and only if $|f|^q\in \weakL^\frac\px q (\Rn)$.
\end{proposition}

\begin{proof}
If $f\in L^\px(\Rn)$, then 
\[
\rho_{\frac{\px}{\qx}}\big(\lambda \chi_{\{|f|^\qx>\lambda\}}\big)
\le 
\rho_{\frac{\px}{\qx}}\big(|f|^\qx\big)
=
\rho_\px(f),
\]
so $|f|^\qx\in \weakL^\frac\px\qx(\Rn)$.

Conversely, let $f$ be such that $|f|^\qx\in \weakL^\frac\px\qx(\Rn)$ for every function $q:\Rn\to (0,\infty)$. Define $q:\Rn\to (0,\infty)$
such that 
\[
(\tfrac12)^\frac1{q(x)}= \tfrac12 \min\{|f(x)|,1\}
\]
for $|f(x)|>0$ and set $q=1$ in $\{f=0\}$.
Let $\lambda = \frac12$ and note that
$\{|f|^\qx>\lambda\} = \{|f|>0\}$. 
Then we find that
\begin{align*}
\rho_{\frac{\px}{\qx}}\big(\lambda \chi_{\{|f|^\qx>\lambda\}}\big)
& =
\int_{\{|f|^\qx>\lambda\}} \lambda^{\frac{p(x)}{q(x)}}\, dx 
=
\int_{\{|f|>0\}} 2^{-p(x)} \min\{|f|,1\}^{p(x)}\, dx\\
& \ge
2^{-p^+}\!\! \int_{\Rn} \min\{|f|,1\}^{p(x)}\, dx
\ge 
2^{-p^+} \rho_\px(f \, \chi_{\{|f|\le 1\}}).
\end{align*}
Hence by the definition of the weak
space we obtain that $\rho_\px(f \, \chi_{\{|f|\le 1\}})$ is finite.

To estimate large values of $f$, let $q:\Rn\to (0,\infty)$ be
such that 
\[
2^\frac1{q(x)}= \tfrac12 \max\{|f(x)|,1\}.
\]
Let $\lambda = 2$ and note that
$\{|f|^\qx>\lambda\} \supset \{f\ge 1\}$. Now by a similar calculation as above, we conclude that 
$\rho_\px(f \, \chi_{\{|f|\ge 1\}})$ is finite. 
Thus $f\in L^\px(\Rn)$.

The last claim, regarding the case of $q$ constant, follows from a change of variables:
\[
\Big(\sup_{\lambda>0} \lambda\, \| \chi_{\{|f|> \lambda\}}\|_\px\Big)^q
= \sup_{\lambda>0} \lambda\, \| \chi_{\{|f|> \lambda^\frac1q\}}\|_\px^q
= \sup_{\lambda>0} \lambda\, \| \chi_{\{|f|^q> \lambda\}}\|_{\frac\px q}. \qedhere
\]
\end{proof}


\section{Strong-to-weak estimates for the Riesz potential}

Let $\alpha:\Omega \to \R$ be log-H\"older continuous with
$0<\alpha^-\leq\alpha^+<n$. We consider the Riesz potential
\[
    \I_{\ax} f(x) := \int_\Omega \frac{|f(y)|}{|x-y|^{n-\alpha (y)}} \, dy
\]
in $\Omega$, and write
\[
p^\#_\alpha (x) := \frac{np(x)}{n-\alpha(x) p(x)}.
\]
Because $\Omega$ is bounded and $\alpha$ is $\log$-H\"older continuous we observe as in \cite[p.\ 270]{Has09} that
$\I_{\ax} f(x)$ and
\[
\I_{\alpha(x)}f(x) = \int_\Omega \frac{|f(y)|}{|x-y|^{n-\alpha(x)}} \, dy
\]
are pointwise equivalent. Thus we obtain the following result from
\cite[Proposition~6.1.6]{DieHHR11}.

\begin{proposition}\label{Jlem}
Let $p \in \PPln (\Omega)$, $\alpha\in\PPzln(\Omega)$ and $(\alpha p)^+<n$. Then
\[
\I_\ax f(x) \lesssim
\left [ \M f (x)\right ]^{1 - \tfrac{\alpha(x)p(x)}{n}}.
\]
for every $f \in L^{p(\cdot)}(\Omega)$ with $\|f\|_{p(\cdot)} \le 1$.
\end{proposition}

Here $\M$ denotes the Hardy-Littlewood maximal function given by
\[
\M f(x) := \sup_{t>0} |f|_{B(x,t)}:= \sup_{t>0}\frac{1}{|B(x,t)|}\int_{B(x,t)} |f(y)| \, dy.
\]
For a measurable function $f$ and measurable set $B$ we use the notation $f_B$ for the mean integral of $f$ over $B$.

We also need the following Jensen-type inequality.
The lemma is a restatement of \cite[Theorem~4.2.4]{DieHHR11}
in our current notation,
cf.\ also the proof of Lemma~4.3.6 in the same source.

\begin{lemma}\label{lem1H}
Let $A\subset \Rn$ be measurable and $p \in \PPln(A)$.
If $f\in L^{p(\cdot)}(A)$ and $\Vert f\Vert_{p(\cdot)} \le 1$, then
\[
(|f|_B)^{p(x)} \lesssim \big(|f|^\px + h\big)_B
\]
for every $x\in A$ and every ball $B\subset A$ containing $x$, where
$h\in \wL^1(A)\cap L^\infty(A)$.
\end{lemma}

The next statement shows that the Riesz potentials behave as expected in the  variable exponent weak space. We will use the exponent $q$ to overcome the difficulty illustrated in Proposition~\ref{prop:notWorking}.

\begin{theorem}\label{thm:R-L1-to-weak}
Suppose that $p \in \PPln(\Omega)$, $\alpha \in \PPzln(\Omega)$ and $(\alpha
p)^+< n$. If $f\in L^\px(\Omega)$, then
$(\I_\ax f)^\qx \in \wL^{p^\#_\alpha(\cdot)/\qx}(\Omega)$ for every $q \in \PPzln(\Omega)$.
\end{theorem}

\begin{proof}
  By \eqref{eq:weak-lp-modular}, it is enough to show that for every
  $f\in L^{p(\cdot)}(\Omega)$ with $\Vert f\Vert_{p(\cdot)} \le 1$ and
  every $t>0$ we have
\[
    \int_{\{(\I_\ax f)^\qx >t\}} t^{p^\#_\alpha (x)/q(x)} dx \lesssim 1.
\]

By Proposition~\ref{Jlem}, for a suitable $c>0$,
\[
\Big\{(\I_\ax f(x))^{q(x)} >t\Big\}
\subset \Big\{c\, [\M f(x)]^{\tfrac{p(x)q(x)}{p^\#_\alpha(x)}} >t\Big\}=:E.
\]
By the definition of the maximal function, for every $x\in E$ we may
choose $B_x:= B(x,r_x)$ such that $\displaystyle
c\,(|f|_{B_x})^{\tfrac{p(x)q(x)}{p^\#_\alpha(x)}} > t$.  Since
$\|f\|_1\lesssim \| f\|_{p(\cdot)} \le 1$ we get $|f|_{B_x} \lesssim
|B_x|^{-1}$.  Denote $r:={pq}/{p^\#_\alpha}$. Then
\[
t\lesssim |f|_{B_x}^{r(x)}
\le
(1+|f|_{B_x})^{r(x)}
= (1+|f|_{B_x})^{r(y)} (1+|f|_{B_x})^{r(x)-r(y)},
\]
where $y \in B_x$. If $r(x)-r(y) \le 0$, then
$(1+|f|_{B_x})^{r(x)-r(y)}\le 1$. If $r(x)-r(y) >0$, then we obtain by
$\log$-H\"older continuity (see \eqref{eq:logH-osc}) that
\[
(1+|f|_{B_x})^{r(x)-r(y)} \le (1+ |B_x|^{-1})^{r(x)-r(y)}
\le 2^{p^+ q^+}\big(1+ |B_x|^{r(y)-r(x)}\big) \lesssim 1.
\]
Hence we have for every $y \in B_x$ that
\[
t\lesssim  (1+|f|_{B_x})^{r(y)}.
\]

By the Besicovitch covering theorem there is a countable covering
subfamily $(B_i)$ of $\{B_x\}$ with bounded overlap.
Thus we obtain by Lemma~\ref{lem1H} that
\[
\begin{split}
\int_E t^{p^\#_\alpha (x)/q(x)} dx
&\le \sum_i \int_{B_i} t^{p^\#_\alpha (x)/q(x)} dx
\lesssim \sum_i \int_{B_i}  (1+|f|_{B_i})^{p(x)} dx\\
&\lesssim \sum_i \bigg(\int_{B_i} |f(y)|^{p(y)}
+h(y) \, dy + |B_i|\bigg)\lesssim 1.
\qedhere
\end{split}
\]
\end{proof}


\section{Real interpolation and weak Lebesgue spaces}

It is well known that real interpolation between the spaces $L^p$ and $L^\infty$ gives a weak Le\-besgue space in the limiting situation when the second interpolation parameter equals $\infty$. We shall prove that the same holds in the variable exponent setting.

We recall that, for $0<\theta <1$ and $0<q\leq \infty$, the
\textit{interpolation space $(A_0,A_1)_{\theta,q}$} is formed from compatible
quasi-normed spaces $A_0$ and $A_1$ by defining a norm as follows.
For $a\in A_0 + A_1$ we set
\[
\| a \|_{(A_0,A_1)_{\theta,q}}:=
\begin{cases}\displaystyle
\bigg( \int_0^\infty \big[ t^{-\theta} K(t,a) \big]^q\,\frac{dt}{t} \bigg)^{1/q} &\text{when } q<\infty, \\
\displaystyle
\sup_{t>0} t^{-\theta} K(t,a)
&\text{when } q=\infty.
\end{cases}
\]
Here the \textit{Peetre $K$-functional} is given by
$$
K(t,a):=K(t,a;A_0,A_1):=
\inf\limits_{a_0+a_1=a \atop a_0\in A_0,a_1\in
A_1} \big( \|a_0\|_{A_0} + t\,\|a_1\|_{A_1} \big), \ \ \ t>0.
$$

We saw in Proposition~\ref{prop:notWorking} that weak $L^\px$-spaces
are not very well behaved. Real interpolation in the variable
exponent setting is even more challenging (cf.\ \cite{AlmH_pp13,KemV_pp13}). 
Fortunately, we can get
quite far with the following special case, whose proof already is
quite complicated.

\begin{theorem}\label{thm:LorentzInfinity}
Let $p\in \PPz(\Rn)$. For $\theta\in(0,1)$,
\[
\big(L^{(1-\theta)\px}(\Rn), L^\infty(\Rn) \big)_{\theta,\infty} = \weakL^\px(\Rn).
\]
\end{theorem}

\begin{proof}
Denote ${p_0}:=(1-\theta) p$ and $X:=\big(L^{p_0(\cdot)}(\Rn), L^\infty(\Rn)\big)_{\theta,\infty}$.
Then by definition
\[
\| f\|_X
=
\sup_{t>0} t^{-\theta} \inf_{f_0+f_1=f} \big(\|f_0\|_{p_0(\cdot)} + t\,\|f_1\|_\infty \big).
\]
We assume without loss of generality that $f, f_0,f_1\ge 0$.

We start by proving that $\|f\|_X \gtrsim \|f\|_\weakLpx$. Let $\lambda>0$
be such that $\| f\|_\weakLpx < 2 \lambda \| \chi_A\|_\px$ where
$A:=\{f> \lambda\}$. Then it remains to prove the second of the inequalities
\[
\|f\|_X
\ge
\|\lambda\chi_A\|_X
\ge
\|\lambda\chi_A\|_\px
\gtrsim
\| f\|_\weakLpx.
\]
Suppose that $f_0+f_1=f$ and that $\|f_1\|_\infty=s$. Then we see that 
\[
\inf_{f_0=f-f_1} \big(\|f_0\|_{p_0(\cdot)} + t\,\|f_1\|_\infty \big)
=
\big\|f-\min\{f,s\}\big\|_{p_0(\cdot)} + t\,s .
\]
Hence in the definition of $\|f\|_X$ we may take the infimum over $s>0$
and functions $f_1:= \min\{f,s\}$, $f_0:=f-f_1$. Thus we calculate
\begin{align*}
\|\chi_A\|_X
&=
\sup_{t>0} t^{-\theta} \inf_{s\in [0,1]} \big( (1-s) \| \chi_A\|_{p_0(\cdot)} + ts\big) \\
&=
 \sup_{t>0} t^{-\theta} \min\{\| \chi_A\|_{p_0(\cdot)}, t \} \\
&=
\, \| \chi_A\|_{p_0(\cdot)}^{1-\theta}
=
\, \| \chi_A\|_\px.
\end{align*}
This completes the proof of the inequality $\|f\|_X \gtrsim \|f\|_\weakLpx$.

We show next that $\|f\|_X \lesssim \|f\|_\weakLpx$. By homogeneity,
it suffices to consider the case where the right hand side equals one.
Thus by \eqref{eq:weak-lp-modular2} we can assume that
\begin{equation}\label{eq:weakThmAssumption}
1\ge
\int_{\{f>\lambda\}} \lambda^{p(x)}\, dx =
\int_{\{f>\lambda\}} \lambda^\frac{p_0(x)}{1-\theta}\, dx =
\int_{\{f>z^{1-\theta}\}} z^{p_0(x)}\, dx
\end{equation}
for every $\lambda >0$.

Since $f_0=f-\min\{f,s\}=\max\{f,s\}-s=\max\{f-s,0\} $, we need to
prove that
\[
\sup_{t>0} t^{-\theta} \inf_{s>0} \big( \|\max\{f-s, 0\}\|_{p_0(\cdot)} + t s\big)\lesssim1.
\]
We choose $s:= t^{\theta-1}$ so that $t^{-\theta} t s = 1$. Thus it suffices to show
that
\[
\|t^{-\theta} \max\{f-t^{\theta-1},0\}\|_{p_0(\cdot)} \lesssim 1
\]
for all $t>0$. We next note that $\max\{f-t^{\theta-1},0\} \le f
\chi_{\{f>z^{1-\theta}\}}$ with $z:=\frac1t$. Thus by
\eqref{eq:mod-norm}, it suffices to show that
\begin{equation}\label{eq:modularForm}
\int_{\{f>z^{1-\theta}\}} (z^{\theta} f)^{p_0(x)}\, dx \lesssim 1
\end{equation}
for all $z>0$. It is enough to show that the inequality holds for all $z = 2^{k_0}$, $k_0 \in \Z$.

Define
\[
A_k := \big\{x\in \Rn \, | \, f(x) \in (2^{k(1-\theta)}, 2^{(k+1)(1-\theta)}] \big\}\,, \quad k\in\Z.
\]
For $z=2^k$, we observe that $A_k \subset \{f>z^{1-\theta}\}$ and 
thus conclude from \eqref{eq:weakThmAssumption} that
\[
\int_{A_k} 2^{k p_0(x)}\, dx \le 1.
\]
Substituting $z=2^{k_0}$ in \eqref{eq:modularForm},
we find that it is enough to prove that
\[
\sum_{k=k_0}^\infty \int_{A_k} \big(2^{k_0\theta} 2^{(k+1)(1-\theta)}\big)^{p_0(x)}\, dx \lesssim 1
\]
for all $k_0\in \Z$. So we estimate
\[
\int_{A_k} \big(2^{k_0\theta} 2^{(k+1)(1-\theta)}\big)^{p_0(x)}\, dx
\le
\big( 2^{(k_0-k)\theta}\big)^{p_0^-} \int_{A_k}  2^{k p_0(x)}\, dx
\le
2^{(k_0-k)\theta p_0^-}.
\]
Hence it follows that
\[
\sum_{k=k_0}^\infty \int_{A_k} \big( 2^{k_0\theta} 2^{(k+1)(1-\theta)}\big)^{p_0(x)}\, dx
\le
\sum_{k=k_0}^\infty 2^{(k_0-k)\theta p_0^-}
=
\frac{1}{1-2^{-\theta p_0^-}}<\infty,
\]
which is the required upper bound. 
\end{proof}

The following feature is the main property of the the real interpolation
method \cite[Proposition~2.4.1]{Tr1}:
If $T$ is a linear operator which is bounded from
$X_0$ to $Y_0$ and from $X_1$ to $Y_1$, then $T$ is bounded from
\[
(X_0,X_1)_{\theta,q}
\text{ to }
(Y_0,Y_1)_{\theta,q}
\]
for $\theta\in (0,1)$ and $q\in (0,\infty]$.  If simple functions are
dense in the spaces, then the claim holds also for sublinear operators
(cf.\ \cite[Theorem~1.5.11]{BruK91}, or \cite[Corollary~A.5]{DieHN04}
for the variable exponent case; see also \cite[Lemma~4.1]{ACae11} for
a discussion in a general framework).  This, together with
Theorem~\ref{thm:LorentzInfinity} for $X_0=Y_0= L^{\p}(\Rn)$ and
$X_1=Y_1= L^{\infty}(\Rn)$ yields the following corollary.

\begin{corollary}
Assume that $T$ is sublinear,
$T: L^\px (\Rn) \to L^\px(\Rn)$ is bounded, and $T: L^\infty(\Rn)  \to L^\infty(\Rn) $ is bounded.
Then $T: \weakL^{\lambda\px}(\Rn)  \to \weakL^{\lambda\px}(\Rn) $ is bounded for every $\lambda>1$.
\end{corollary}

L.\ Diening has shown that the boundedness of $\M: L^\px (\Rn) \to L^\px(\Rn)$
implies the boundedness of $\M: L^{s \px} (\Rn) \to L^{s \px}(\Rn)$ for some
$s<1$ \cite[Theorem~5.7.2]{DieHHR11}. Furthermore, it is known that
the maximal operator is bounded on $L^\px(\Rn)$ when
$p\in \PPln(\Rn) $ and $p^->1$ \cite[Theorem~4.3.8]{DieHHR11}. In view of the
previous result these facts immediately imply Theorem~\ref{thm:strongWeakM}.


\section{Weak-to-weak estimates for the Riesz potential}

As usual, we denote by $p'$
the H\"older conjugate exponent of $p$, taken in a
point-wise sense, $1/p(x) + 1/p'(x) =1$. Following Diening (and \cite{DieHHR11}),
for exponents we use the notation $p_B$ to denote the
\textit{harmonic} mean of $p$ over the measurable set $B$,
\[
p_B := \bigg(\fint_B \frac1{p(x)} \, dx\bigg)^{-1}.
\]
The following claim is proved as part of the proof of
\cite[Lemma~6.1.5]{DieHHR11}.

\begin{lemma}\label{lem:norm}
Let $p\in \PPln(\Rn)$ with $1< p^-\le p^+ < \frac{n}{\alpha}$ for $\alpha\in (0,n)$. Then
\[
\big\| \, |x-\cdot|^{\alpha-n} \chi_{\Rn\setminus B} \big\|_{L^{p'(\cdot)}(\Rn)} \approx
|B|^{-\frac1{(p_\alpha^\#)_B}}
\]
where $B$ is a ball centered at $x\in \Rn$.
\end{lemma}

We next generalize this claim to slightly more general norms, which will
appear below when we estimate in the dual of a weak Lebesgue space.
We need the following auxiliary result.

\begin{lemma}\label{lem:infimum}
For $\alpha,\beta, \delta, t>0$,
\[
\inf_{R\in [\delta,\infty)}
\Big(\tfrac{\alpha}\beta R^{-\beta} + t\, (R^\alpha - \delta^\alpha) \Big)
\approx
\min\Big\{t^{\frac{\beta}{\alpha + \beta}}, \delta^{-\beta} \Big\},
\]
and the infimum occurs at $R<1$ if and only if $t>1$ and $\delta<1$. 
\end{lemma}

\begin{proof}
Denote $f(R):= \tfrac{\alpha}\beta R^{-\beta} + t\, (R^\alpha - \delta^\alpha)$.
Then $f'(R)=- \alpha R^{-\beta-1} + t \alpha R^{\alpha-1}$, which equals zero when
$R = t^{-1/(\alpha+\beta)} =: R_0$. This is a minimum in $(0, \infty)$, since
$f\to \infty$ at $0$ and $\infty$.
When $R=R_0\ge \delta$, we estimate $0\le t\, (R^\alpha - \delta^\alpha)
\le t R^\alpha = R^{-\beta}$. Hence we conclude that
$f(R_0) \approx R_0^{-\beta} = t^{\beta/(\alpha+\beta)}$.
Also note that the unconstrained minimum occurs for
$R<1$ if and only if $t>1$.

However, if $R_0<\delta$, then the constrained minimum
occurs at $\delta$, in which case $f(\delta) = \tfrac{\alpha}\beta
\delta^{-\beta} \approx \delta^{-\beta}$.
Hence the estimate of the minimum equals
\[
t^{\beta/(\alpha+\beta)} \chi_{\{t^{-1/(\alpha+\beta)} \ge \delta\}}
+
\delta^{-\beta} \chi_{\{t^{-1/(\alpha+\beta)} < \delta\}}
=
\min\Big\{t^{\frac{\beta}{\alpha + \beta}}, \delta^{-\beta} \Big\}.
\qedhere
\]
\end{proof}

\begin{lemma}\label{lem:interpolationNormEstimate}
Let $p\in \PPln(\Rn)$ with $1<p^-\le p^+ < \frac{n}{\alpha}$ for $\alpha\in (0,n)$ and
let $\theta>0$ be so small that the infimum of $r:=(1-\theta)p$ is greater than $1$. Then
\[
\big\| \, |x-\cdot|^{\alpha-n}\chi_{\Rn\setminus B}\big\|_{(L^{r'(\cdot)}(\Rn),L^1(\Rn))_{\theta,1}}
\approx
|B|^{-\frac1{(p^\#_\alpha)_B}}
\]
where $B$ is a ball centered at $x\in \Rn$.
\end{lemma}

\begin{proof}
Let $B:=B(x,\delta)$ and denote $f(y):= |x-y|^{\alpha-n}\chi_{\Rn\setminus B}(y)$.
By the definition of the interpolation norm,
\[
\| f\|_{(L^{r'(\cdot)},L^1)_{\theta,1}}
=
\int_0^\infty t^{-\theta} \inf_{f_1+f_2=f} \big(\|f_1\|_{r'(\cdot)} +  t\, \|f_2\|_1\big) \frac{dt}t.
\]
Suppose that $f_1+f_2=f$ and denote $A:= \{ |f_1| \ge \frac12 |f|\}$. 
Then $|f_1| \ge \frac12 |f| \chi_A$ and $|f_2| \ge \frac12 |f| \chi_{\Rn\setminus A}$, so that 
\[
\|f_1\|_{r'(\cdot)} +  t\, \|f_2\|_1
\ge
\tfrac12 \|f \chi_A \|_{r'(\cdot)} +  \tfrac12 t\, \|f \chi_{\Rn\setminus A}\|_1.
\]
Hence
\[
\inf_{f_1+f_2=f} \big(\|f_1\|_{r'(\cdot)} +  t\, \|f_2\|_1\big)
\ge 
\tfrac12 \inf_{A\subset \Rn} \big(\|f \chi_A \|_{r'(\cdot)} + t\, \|f \chi_{\Rn\setminus A}\|_1 \big).
\]
On the other hand the opposite inequality holds with constant $1$, 
since we may choose $f_1=f\chi_A$ and $f_2=f\chi_{\Rn\setminus A}$ 
in the first infimum. So we conclude that
\[
\inf_{f_1+f_2=f} \big(\|f_1\|_{r'(\cdot)} +  t\, \|f_2\|_1\big)
\approx
\inf_{A\subset \Rn} \big(\|f \chi_A \|_{r'(\cdot)} + t\, \|f \chi_{\Rn\setminus A}\|_1 \big).
\]
Since $r'>1$, the infimum is not achieved when $\sup \{|f| \chi_A\} > \inf \{|f| \chi_{\Rn\setminus A} \}$ (since in this case we can shift 
mass to decrease the $L^{r'(\cdot)}$-norm while conserving the $L^1$-norm). Assuming that $|\{f=c\}|=0$ for all $c\in \R$, it follows that 
$A$ must be of the form $\{ |f| < c\}$ for some $c\ge 0$. 
In our case, $f$ is radially decreasing and so $A = \Rn\setminus B(R)$, for some $R\in [\delta,\infty]$. This corresponds to the functions
$f_1=|x-\cdot|^{\alpha-n}\chi_{\Rn\setminus B(R)}$ and
$f_2=|x-\cdot|^{\alpha-n}\chi_{B(R)\setminus B}$. 

For simplicity we denote
$s:=r^\#_\alpha$. A straight calculation gives $\|f_2\|_1 \approx
R^\alpha - \delta^\alpha$.  Then it follows from Lemma~\ref{lem:norm}
that
\[
\| f\|_{(L^{r'(\cdot)},L^1)_{\theta,1}}
\approx
\int_0^\infty t^{-\theta} \inf_{R\in [\delta,\infty)} \Big(R^{-\frac n{s_{B(R)}}} +
t\, (R^\alpha - \delta^\alpha)\Big) \frac{dt}t.
\]
By \cite[Corollary~4.5.9]{DieHHR11}, $R^{-\frac n{s_{B(R)}}} \approx R^{-\frac n{q}}$
where $q:=s_\infty$ if $R\ge 1$ and $q:=s(x)$ otherwise.
Recall that $s_\infty$ is the limit value of $s$ at infinity,
from the definition of $\log$-H\"older continuity.

We further observe that
\[
\inf_{R\in [\delta,\infty)} \Big(R^{-\frac nq} +
t\, (R^\alpha - \delta^\alpha)\Big)
\approx
\inf_{R\in [\delta,\infty)} \Big(\tfrac{\alpha q}n R^{-\frac nq} +
t\, (R^\alpha - \delta^\alpha)\Big),
\]
since $\frac{\alpha q}n$ is bounded away from $0$ and infinity.
Then we apply Lemma~\ref{lem:infimum} twice,
for $\beta=\frac n {s_\infty}$ and $\beta=\frac n {s(x)}$,
to conclude that
\[
\inf_{R\in [\delta,\infty]} \Big(\tfrac{\alpha q}n R^{-\frac n q} + t\, (R^\alpha - \delta^\alpha) \Big)
\approx
\min\Big\{t^{\frac{n}{n+\alpha q}}, \delta^{-\frac n q} \Big\},
\]
where $q:=s_\infty$ if and only if $t\ge 1$ and $\delta \le 1$
and $q:=s(x)$ otherwise.

Let $t_0>0$ be such that $t_0^\frac{n}{n+\alpha q} = \delta^{-\frac n{s_B}}$.
Then
\[
\| f\|_{(L^{r'(\cdot)},L^1)_{\theta,1}}
\approx
\int_0^{t_0} t^{-\theta-1 + \frac{n}{n+\alpha q}}\, dt
+
\delta^{-\frac n{s_B}} \int_{t_0}^\infty t^{-\theta-1}\, dt.
\]
If $t_0>1$ (so that $\delta<1$) we find that
\[
\int_0^{t_0} t^{-\theta-1 + \frac{n}{n+\alpha q}}\, dt
=
\int_0^1 t^{-\theta-1 + \frac{n}{n+\alpha s_\infty}} dt
+
\int_1^{t_0} t^{-\theta-1 + \frac{n}{n+\alpha s(x)}} dt
\approx
t_0^{-\theta + \frac{n}{n+\alpha s(x)}}.
\]
So in this case
\[
\| f\|_{(L^{r'(\cdot)},L^1)_{\theta,1}}
\approx
\Big(
t_0^{-\theta + \frac{n}{n+\alpha s(x)}} + \delta^{-\frac n{s_B}} t_0^{-\theta} \Big)
\approx
\delta^{-\frac n{s_B} + \theta \frac n{s_B}\frac{n+\alpha s(x)}{n}}.
\]
Since $p$ is $\log$-H\"older continuous and $x\in B=B(x,\delta)$, we have
$\delta^{s_B}\approx \delta^{s(x)}$. Thus
\[
\delta^{-\frac n{s_B} + \theta \frac n{s_B}\frac{n+\alpha s(x)}{n}}
=
\delta^{(\theta \alpha s(x) + (\theta - 1)n)\frac 1{s_B}}
\approx
\delta^{\theta \alpha + (\theta - 1)n\frac 1{s_B}}
=
\delta^{-\frac{n}{p^\#_B}}.
\]

For $t_0\le 1$ we similarly conclude that
$\displaystyle
\| f\|_{(L^{r'(\cdot)},L^1)_{\theta,1}}
\approx
\delta^{-\frac {n}{p^\#_B}},
$
using that $\delta^{s_B}\approx \delta^{s_\infty}$ which holds
by the $\log$-H\"older decay since $\delta\ge 1$. \qed
\end{proof}

According to \cite[Theorem~1.11.2]{Tri78} the duality formula
\[
\big((A_0,A_1)_{\theta,q}\big)^* = (A_0^*,A_1^*)_{\theta,q'}
\]
holds when $q\in [1,\infty)$ and $A_0 \cap A_1$ is dense both in $A_0$ and in $A_1$.
We choose $A_0=L^{p'(\cdot)}(\Rn)$, $A_1=L^1(\Rn)$ and $q=1$. Then we obtain
\[
(L^{p'(\cdot)},L^1)_{\theta,1}^* = (L^\px,L^\infty)_{\theta,\infty}.
\]
Hence we obtain the H\"older inequality
\[
\int_\Rn f(x) g(x) \, dx \lesssim \|
f\|_{(L^\px,L^\infty)_{\theta,\infty}} \|
g\|_{(L^{p'(\cdot)},L^1)_{\theta,1}}.
\]

In the following result we generalize \cite[Lemma~6.1.5]{DieHHR11} where the
same conclusion was reached under the stronger assumption that
 $\|f\|_{L^\px}\le 1$.

\begin{lemma}
  \label{lem:rieszvsM2}
  Let $p\in \PPln(\Rn)$ with $1<p^- \le p^+ < \frac{n}{\alpha}$ for
  $\alpha\in (0,n)$. Let $x\in \Rn$, $\delta>0$, and $f \in
  \wL^{p(\cdot)}(\Rn)$ with $\|f\|_{\weakL^\px}\le 1$. Then
  \begin{align*}
    \int_{\Rn \setminus B(x,\delta)}
    \frac{|f(y)|}{|x-y|^{n-\alpha}} \,dy
    \lesssim
    |B(x,\delta)|^{-\frac{1}{(p^\#_\alpha)_{B(x,\delta)}}}.
  \end{align*}
\end{lemma}

\begin{proof}
Set $B:=B(x,\delta)$ and $r:=(1-\theta)p$, where $\theta>0$ is so small that
$r^->1$. By Theorem~\ref{thm:LorentzInfinity} we have $(L^{r(\cdot)},L^\infty)_{\theta,\infty}= \wL^\p$ and thus by
H\"{o}lder's inequality, the assumption $\|f\|_{\wL^\p}\le 1$ and
Lemma~\ref{lem:interpolationNormEstimate} we obtain that
\[
      \int_{\Rn \setminus B}
      \frac{\abs{f(y)}}{\abs{x-y}^{n-\alpha}} \,dy
      \lesssim \| f\|_{(L^{r(\cdot)},L^\infty)_{\theta,\infty}}
\big\|  \abs{x-\cdot}^{\alpha-n}\big\|_{(L^{r'(\cdot)},L^1)_{\theta,1}}
      \lesssim |B|^{-\frac{1}{(p^\#_\alpha)_{B}}}. 
\qedhere
\]
\end{proof}

With this result we immediately obtain a generalization of \cite[Lemma~6.1.8]{DieHHR11}
as follows, where similarly the condition $\norm{f}_{p(\cdot)} \leq 1$ has been replaced
by $\norm{f}_{\weakL^p(\cdot)} \leq 1$:

\begin{lemma}\label{lem:riesz-vs-M}
Let $p\in \PPln(\Rn)$ with $1<p^- \le p^+ < \frac{n}{\alpha}$ for $\alpha\in (0,n)$.
Then
  \begin{align*}
    \I_\alpha f(x)^{p^\#_\alpha(x)}
    &\lesssim \M f(x)^{p(x)} + h(x),
  \end{align*}
  for all $x\in \Rn$, and $f \in \wL^{p(\cdot)}(\Rn)$ with $\norm{f}_{\weakL^\px} \leq 1$, where $h \in \weakL^1(\Rn) \cap L^\infty(\Rn)$ is positive.
The implicit constant and $h$ depend only on
$\log$-H\"{o}lder constant of $p$, $p^-$, $p^+$, $\alpha$, and $n$.
\end{lemma}

Then we obtain the following analogue of \cite[Theorem~6.1.9]{DieHHR11}
using the previous lemma and Theorem~\ref{thm:strongWeakM}:

\begin{theorem} \label{thm:I-is-bdd-weak}
Let $p \in \PPln(\Rn)$, $\alpha\in \PPzln(\Rn)$ and $1<p^- \le
p^+ < \frac n{\alpha^+}$. If $f\in \weakL^{\px}(\Rn)$ and
$|f|^{p(\cdot)/p^\#_\alpha(\cdot)}\in \weakL^{p^\#_\alpha
(\cdot)}(\Rn)$, then the function $x\mapsto \I_{\alpha(x)} f(x)$ belongs to
$\weakL^{p^\#_\alpha(\cdot)}(\Rn)$.
\end{theorem}

\begin{proof}
We write $t:= p/p^\#_\alpha$. By a scaling argument we may assume
that $\norm{f}_{\weakL^p(\cdot)} \leq 1$. By
Lemma~\ref{lem:riesz-vs-M}, there exists $h \in \weakL^1(\Rn) \cap L^\infty(\Rn)$
such that $(\I_{\alpha(x)} f(x))^{p^\#_\alpha(x)} \le c\, \big(\M f(x)^{p(x)} + h(x)\big)$.
Then
\[
\{ \I_{\alpha(x)} f(x) > \lambda\} \subset \{\M f(x)^{t(x)} > c\,
\lambda \} \cup \{h^{1/p^\#_\alpha(x)} > c\, \lambda \}
\]
and so we obtain
\[
\int_{\{ \I_{\alpha(x)} f(x) > \lambda\}}\lambda^{p^\#_\alpha(x)}\,
dx \le \int_{\{ \M f(x)^{t(x)} > c\,
\lambda\}}\lambda^{p^\#_\alpha(x)}\, dx + \int_{\{h^{1/p^\#_\alpha(x)}> c\, \lambda\}}
\lambda^{p^\#_\alpha(x)}\, dx.
\]
Now $h\in \weakL^1(\Rn) \cap L^\infty(\Rn) \subset
L^{p^\#_\alpha(\cdot)}(\Rn)$, so the last term is bounded. Thus
it remains to show that $(\M f)^{t(\cdot)}\in
\weakL^{p^\#_\alpha(\cdot)}(\Rn)$.

Let $t_0\in (1/(p^\#_\alpha)^-,t^-)$.  Since $|f|^{t(\cdot)}\in
\weakL^{p^\#_\alpha(\cdot)}(\Rn)$, we obtain that $|f|^{t(\cdot)/t_0}\in
\weakL^{t_0 p^\#_\alpha(\cdot)}(\Rn)$. By assumption,
$(t_0p^\#_\alpha)^- >1$ and hence it follows from Theorem~\ref{thm:strongWeakM}
that $\M(|f|^{t(\cdot)/t_0})\in
\weakL^{t_0p^\#_\alpha(\cdot)}(\Rn)$. Since $t/t_0\ge 1$, 
Lemma~\ref{lem1H} implies that $(\M f)^{t(\cdot)/t_0}\in
\weakL^{t_0p^\#_\alpha(\cdot)}(\Rn)$, and thus $(\M f)^{t(\cdot)}\in
\weakL^{p^\#_\alpha(\cdot)}(\Rn)$. \qed
\end{proof}

As was noted before, $\I_{\alpha(x)} f(x) \approx \I_\ax f(x)$ in
bounded domains. Furthermore, a $\log$-H\"older continuous exponent
in a domain can be extended to a variable exponent 
in the whole space, with the same
parameters \cite[Proposition~4.1.7]{DieHHR11}.
Thus we obtain the following corollary.

\begin{corollary} \label{cor:I-is-bdd-weak}
Let $p \in \PPln(\Omega)$, $\alpha\in \PPzln(\Omega)$, $p^- >1$
and $(\alpha p)^+ < n$.
If $f\in \weakL^{\px}(\Omega)$ and
$|f|^{p(\cdot)/p^\#_\alpha(\cdot)}\in \weakL^{p^\#_\ax}(\Omega)$, then
$\I_\ax f \in \weakL^{p^\#_\alpha(\cdot)}(\Omega)$.
\end{corollary}

Note that a direct use of Theorem~\ref{thm:I-is-bdd-weak} leads to the assumption
$\alpha^+ p^+ < n$ in the corollary. However, $(\alpha p)^+ < n$ if and only
if the domain can be split into a finite number of parts in each of which the
inequality $\alpha^+ p^+ < n$ holds, so in fact these conditions are equivalent.


\section{The Wolff potential}
\label{sect:wolff}

Let $\mu$ be a positive, locally finite Borel measure.
The (truncated) Wolff potential is defined by
\[
  \W^{\mu}_{\alpha,p}(x,R):=\int_0^R\left(\frac{\mu(B(x,r))}{r^{n-\alpha p}}
  \right)^{1/(p-1)}\frac{dr}{r};
\]
with the full Wolff potential being $\W^{\mu}_{\alpha,p}(x):=
\W^{\mu}_{\alpha,p}(x,\infty)$. There are several ways in which this
can be generalized to the variable exponent setting. The most
straigth-forward is to consider the point-wise potential $x\mapsto
\W^{\mu}_{\alpha(x),p(x)}(x,R)$.

In this case we immediately obtain the following inequality from the
constant exponent setting:
\[
\W^{\mu}_{\alpha(x),p(x)}(x) \lesssim \I_{\alpha(x)}\Big( \I_{\alpha(x)}\mu^\frac1{p(x)-1} \Big)(x).
\]
This was observed in \cite[Subsection~5.2]{BogH10}. As we have noted,
in the bounded domain case the Riesz potentials $\I_{\alpha(x)}f(x)$ and $\I_\ax f(x)$ are comparable. Thus we obtain that
\[
\W^{\mu}_{\alpha(x),p(x)}(x,R) \lesssim \I_{\ax}\Big( \I_{\ax}\mu^\frac1{p(x)-1} \Big)(x).
\]
However, there is no immediate way to change the exponent $\frac1{p(x)-1}$. As far as we can see, the above inequality \textit{cannot} be used to derive Theorem~\ref{thm:SOLA-weak}, thus the validity of the claims in this part of \cite[Section~5.2]{BogH10} are in doubt. (Additionally, their claim that $\I_\ax : L^\px(\Rn) \to L^{p^\#_\alpha(\cdot)}(\Rn)$
is bounded is false, see \cite[Example~4.1]{Has09}; the claim only holds for bounded domains. Of course, the latter claim is what is actually needed.)

The Wolff potential has also been studied by F.-Y.\ Maeda
\cite{Mae09}. To state the result as clearly as possible, let us
denote $g(y):= \I_\alpha f(y)^\frac1{p(y)-1}$. Maeda proved that
\[
\W^{\mu}_{\alpha,p(x)}(x) \lesssim \I_{\alpha} g(x)
\]
Since $\I_{\alpha(x)}f(x) \approx \I_\ax f(x)$, this implies the
desired inequality, which can succinctly be stated as
\begin{equation}\label{eq:wolff-pointw}
\W^{\mu}_{\alpha(x),p(x)}(x) \lesssim \I_{\ax}\Big( \I_{\ax}\mu^\frac1{\px-1} \Big)(x),
\end{equation}
provided one keeps track of which dot is related to which operation. The right
hand side in this equation is called the Havin--Maz'ya
potential which is denoted by $\V^\mu_{\ax, \p}(x)$.

The following result is now a consequence of Theorems
\ref{thm:R-L1-to-weak} and \ref{thm:I-is-bdd-weak}, and \eqref{eq:wolff-pointw}.

\begin{theorem} \label{thm:wolff}
  Let $\alpha$, $r$, and $p$
  be bounded and $\log$-H\"older continuous, with $p^->1$ and $r^-\ge 1$, $0<\alpha^-\leq
  \alpha^+<n$, $(\alpha p r)^+<n$, and $p(x) \ge 1+ 1/r(x)-\alpha(x)/n$ for every $x\in \Omega$. If $f\in L^{r(\cdot)}(\Omega)$, then
  \[x\mapsto \W^f_{\alpha(x), p(x)}(x) \in \wL^{\frac{nr(\cdot)(\p-1)}{n-\ax\p r(\cdot)}}(\Omega).\]
\end{theorem}

\begin{proof}
  By \eqref{eq:wolff-pointw}, it suffices to consider the
  Havin--Maz'ya potential $\V^f_{\ax, \p}$ instead of the Wolff
  potential. Denote $s:= \frac{nr(p-1)}{n-\alpha r}$; by assumption
  $p\ge 1+ 1/r-\alpha/n$ so that $s^-\ge 1$. Choosing $q:=1/(p-1)$ in
  Theorem~\ref{thm:R-L1-to-weak}, we see that
  \[
    (\I_\ax f)^{1/(\p-1)}\in \wL^{s(\cdot)}(\Omega).
  \]
Since $(\alpha p r)^+<n$, we find that $(\alpha s)^+ <n$ and thus by choosing $q:=  s/[(p-1) s^\#_\alpha]$ in Theorem~\ref{thm:R-L1-to-weak} we obtain that
 \[
    \big[(\I_\ax f)^{1/(\p-1)}\big]^{\frac{s(\cdot)}{s^\#_\alpha(\cdot)}}\in \wL^{\frac{ (\p-1) s^\#_\alpha(\cdot) \frac{nr(\cdot)}{n-\ax r(\cdot)}}{s(\cdot)}}(\Omega)= \wL^{s^\#_\alpha(\cdot)}(\Omega).
  \]
Further, since $(\alpha p r)^+<n$, we can use Corollary~\ref{cor:I-is-bdd-weak} for the function $(\I_\ax f)^{1/(\p-1)}$ to conclude that
  \[
    \V^f_{\ax, \p}\in \wL^{s^\#_\alpha(\cdot)}(\Omega).
  \]
  The claim follows from this since $s^\#_\alpha = \frac{nr(p-1)}{n-\alpha p r}$. 
\end{proof}


\section{An application to partial differential equations}
\label{sect:pde}

In this section, we discuss consequences of our results and pointwise
potential estimates for solutions to the nonlinear elliptic equation
\begin{equation}\label{eq:px-laplace-measure}
  -\dive(\abs{\nabla u}^{p(x)-2}\nabla u)=\mu,
\end{equation}
where $\mu$ is a Borel measure with finite mass. The right quantity
for estimating solutions to \eqref{eq:px-laplace-measure} and their
gradients is the Wolff potential $\W^\mu_{\alpha(x),p(x)}(x)$.

Recall that for right hand side data a Borel measure $\mu$ with
finite mass or a function in $L^1$, we use the notion of solutions
obtained as limits of approximations, SOLAs for short. Gradient
potential estimates for SOLAs follow by working with a priori more
regular solutions, and then transferring the information obtained to the
limit. In the case of general measures, the latter step requires some
care, as the approximants converge only in the sense of weak
convergence of measures. For this reason, the approximation argument
is not done using the final potential estimates. Certain intermediate
estimates, from which the actual potential estimates are then built,
need to be used instead.  See \cite[Proof of Theorem 1.4,
p. 668]{BogH10} for details.

An alternative point of view is to start with the fundamental objects
of the nonlinear potential theory related to the $\p$-Laplacian, namely 
$\p$-superharmonic functions. See \cite[Definition 2.1, p. 1068]{LMM}
for the exact definition of this class. For a $\p$-superharmonic
function $u$, there exists a measure $\mu$ such that
\eqref{eq:px-laplace-measure} holds. This is the \emph{Riesz measure}
of $u$. Important results in nonlinear potential theory are derived by
employing measure data equations like
\eqref{eq:px-laplace-measure}. The leading example is the necessity of
the celebrated Wiener criterion for boundary regularity, see
\cite{KM2}.

The gradient potential estimates in \cite{BogH10} are local: one works
in a fixed ball, compactly contained in $\Omega$. Thus the solution
under consideration can be a local SOLA, i.e.\ it suffices to
choose approximations in a fixed compact subset of $\Omega$.

 If $\mu$ is a signed measure, we use the
notation
\[
  \W^{\mu}_{\alpha(x),p(x)}(x,R)=\int_0^R\left(
    \frac{\abs{\mu}(B(x,r))}{r^{n-\alpha(x)p(x)}}
  \right)^{1/(p(x)-1)}\frac{\ud r}{r},
\]
where $\abs{\mu}$ is the total variation of $\mu$.

To extend the gradient potential estimate to $\p$-superharmonic
functions, we need the fact that these functions are local SOLAs. This
is the content of the following theorem.

\begin{theorem}\label{thm:local-sola}
Let $p$ be $\log$-H\"older continuous with $p^-\ge 2$.
  Let $u$ be a $\p$-superharmonic function in a domain $\Omega$ and let $\mu$ be the measure such that
  \begin{displaymath}
    -\dive(\abs{\nabla u}^{p(x)-2}\nabla u)=\mu.
  \end{displaymath}
  For every subdomain $\Omega'\Subset \Omega$ there are sequences of
  solutions $(u_i)$ and smooth, positive functions $(f_i)$ such that
  \begin{displaymath}
    -\dive(\abs{\nabla u_i}^{p(x)-2}\nabla u_i)=f_i \quad\text{in}\quad\Omega',
  \end{displaymath}
  $u_i\to u$ in $W^{1,\q}(\Omega')$ for any continuous $q$ such that
  $q(x)<\frac{n}{n-1}(p(x)-1)$ for all $x\in \overline{\Omega'}$, and
  $f_i\to \mu$ in the sense of weak convergence of measures.
\end{theorem}
\begin{proof}
  This follows in the same way as in the constant exponent case,
  Theorem 2.7 in \cite{KLP}. For the reader's convenience, we sketch
  the argument here with the appropriate references for various
  auxiliary results. The proof consists of two main steps. First, we
  prove the claim when $u$ is a weak supersolution. The general case
  is then reduced to the case of supersolutions by an approximation
  argument using the obstacle problem.

  Assume first that $u$ is a weak supersolution. Then $u\in
  W^{1,\p}_{\loc}(\Omega)$, and the fact that $\mu$ belongs to the
  dual space $\big(W^{1,p(\cdot)}_0(\Omega')\big)^*$ follows from the equation satisfied by
  $u$. Then the case of supersolutions follows by arguing as in
  \cite[Lemma 2.6]{KLP} and using the elementary inequalities between
  the $p$-modular and the Luxemburg norm.

  In the general case, the fact that $u\in W^{1,\q}(\Omega')$ follows
  by a refinement of \cite[Theorem 4.4]{L}. By \cite[Theorem
  6.5]{HHKLM}, we may choose a sequence $(\widetilde{u}_i)$ of
  continuous weak supersolutions increasing to $u$. Arguing as in
  \cite[proof of Theorem 5.1]{HHKLM} we can show that
  $\nabla\min(\widetilde{u}_i,k)\to \nabla \min(u,k)$ pointwise almost
  everywhere for any $k\in \R$. It follows that $\nabla
  \widetilde{u}_i\to \nabla u$ pointwise a.e., and the pointwise
  convergences easily imply that $\widetilde{u}_i\to u$ in
  $W^{1,\q}(\Omega')$. The proof is completed by applying the case of
  supersolutions to the functions $\widetilde{u}_i$, together with the
  convergence of $\widetilde{u}_i\to u$ in $W^{1,\q}(\Omega')$, see
  the proof of Theorem 2.7 in \cite{KLP}. 
\end{proof}

The following pointwise potential estimates hold for local SOLAs and
$\p$-su\-per\-har\-mo\-nic functions. See \cite{BogH10} for
\eqref{eq:pointwise-solution} and \eqref{eq:pointwise-gradient} in the
case of SOLAs, and \cite{LMM} for \eqref{eq:pointwise-solution} for
$\p$-su\-per\-har\-mo\-nic functions. Finally, the gradient estimate
\eqref{eq:pointwise-gradient} holds also for
$\p$-su\-per\-har\-mo\-nic functions by an application of Theorem
\ref{thm:local-sola}. The H\"older continuity of $p$ is required for
the gradient estimate, since its proof uses H\"older estimates for the
gradient of weak solutions (cf.\ \cite{AceM01}).

\begin{theorem}\label{thm:potential-estimates}
  Let $p$ be $\log$-H\"older continuous with $p^-\ge 2$. Let $u$ be
  positive $\p$-super\-har\-mo\-nic or a local SOLA to
  \eqref{eq:px-laplace-measure}.  Then there exists $\gamma>0$ such
  that
  \begin{equation}
    \label{eq:pointwise-solution}
    \abs{u(x_0)}\lesssim \left(\vint_{B(x_0,R)}\abs{u}^\gamma\ud x\right)^{1/\gamma}
    +\W^{\mu}_{1,p(x_0)}(x_0,2R)+R
  \end{equation}
  for all sufficiently small $R>0$.
  For positive $\p$-superharmonic functions, the assumption $p^->1$
  suffices instead of $p^- \ge 2$.

  Suppose next that $p$ is H\"older continuous. Then
  \begin{equation}
    \label{eq:pointwise-gradient}
    \abs{\nabla u(x_0)}\lesssim  \vint_{B(x_0,R)}\abs{\nabla u}\, dx
    +\W^{\mu}_{1/p(x_0),p(x_0)}(x_0,2R)+R
  \end{equation}
  for all sufficiently small $R>0$.
\end{theorem}

The restriction $p^-\geq 2$ in the gradient estimates is
related to the fact that there are substantial differences in gradient
potential estimates in the cases $p<2$ and $p>2$ even with constant
exponents, see \cite{DM2}.  For simplicity, we focus on the prototype
case \eqref{eq:px-laplace-measure} here, but this result, and hence
also Theorem \ref{thm:SOLA-weak} below, hold for more general
equations of the form
\[
  -\dive(a(x,\nabla u))=\mu
\]
under appropriate structural assumptions on $a(x,\xi)$. The interested
reader may refer to \cite{BogH10,LMM} for details.

The following result is an immediate consequence of
Theorems~\ref{thm:wolff} and \ref{thm:potential-estimates}.

\begin{theorem}\label{thm:SOLA-weak}
  Let $p$ and $r$ be $\log$-H\"older continuous with $p^-\ge 2$. Let
  $u$ be a positive $\p$-superharmonic function, or a local SOLA to
  \eqref{eq:px-laplace-measure}, with $\mu \in L^{r(\cdot)}(\Omega)$.
  \begin{itemize}
  \item[(a)]
    If $(pr)^+<n$ and $p \ge 1 + 1/r -1/n$ for every $x\in \Omega$, then
    \[
    u\in \wL_{\loc}^{\frac{nr(\cdot)(\p-1)}{n-\p r(\cdot)}}(\Omega).
    \]
    For positive $\p$-superharmonic functions, the assumption $p^->1$
    suffices.

  \item[(b)] Suppose in addition that $p$ is H\"older continuous.
    If $r^+<n$ and $p\ge 1 + 1/r -1/(np)$ for every $x\in \Omega$, then
    \[
    \abs{\nabla u}\in \wL_{\loc}^{\frac{n r(\cdot) (\p-1)}{n- r(\cdot)}}(\Omega).
    \]
  \end{itemize}
  If $r\equiv 1$, $\mu$ can be a measure with finite mass instead of a
  function.  Each of the inclusions comes with an explicit estimate.
\end{theorem}

Theorem \ref{thm:SOLA-weak0} is of course contained in the above
theorem when $r\equiv 1$.  
The interesting case in these results is when $r^- =1$; if
$r^->1$, we can use the pointwise inequality \eqref{eq:wolff-pointw}
and the strong-to-strong estimate for the Riesz potential to get
estimates in strong Lebesgue spaces with the same exponents.

When $r\equiv 1$, the above inclusions are sharp for constant $p$ on
the scale of $\wL^q$ spaces. This is a special case of the following
examples.

\begin{example}
  Let $B$ be the unit ball in $\R^n$, and assume that the exponent $p$ is smooth
  and radial. Define the function $u$ by
  \begin{equation}
    \label{eq:fundamental-solution}
    u(x):=\int_{\abs{x}}^1(p(\rho)\rho^{n-1})^{-1/(p(\rho)-1)}\ud \rho.
  \end{equation}
    Then by \cite[Section 6]{HHKLM} $u$ is $\p$-superharmonic in $B$,
  and Theorem 4.10 of \cite{L} implies that
  \[
    -\dive(\abs{\nabla u}^{p(x)-2}\nabla u)=K\delta,
  \]
  where $K>0$ and $\delta$ is Dirac's delta at the origin. The exact
  value of $K$ is not important.

  Assume that $q$ is $\log$-H\"older continuous. We will
  show that $u\in \wL^{\q}(B)$ if and only if
  \begin{equation}\label{eq:q-cond}
    q(0)\leq \frac{n (p(0)-1)}{n-p(0)}
  \end{equation}
  and $\abs{\nabla u}\in \wL^{\q}(B)$ if and only if
  \begin{equation}
    \label{eq:q-cond-gradient}
    q(0)\leq \frac{n(p(0)-1) }{n-1} .
  \end{equation}

  We reason as follows to get these characterizations. First,
  $\log$-H\"older continuity of $p$ implies that
  \begin{equation}
    \label{eq:fund-estimate}
    \abs{u(x)}\approx \abs{x}^{-\frac{n-p(0)}{p(0)-1}}
  \end{equation}
  and
  \begin{equation}
    \label{eq:fund-gradient}
    \abs{\nabla u(x)}\approx \abs{x}^{-\frac{n-1}{p(0)-1}}.
  \end{equation}
  The inclusions
  \begin{equation}\label{eq:incl}
    \Big\{t<c^{-1}\abs{x}^{-\frac{n-p(0)}{p(0)-1}}\Big\}
    \subset\Big\{t<u(x)\Big\}\subset\Big\{t<c\abs{x}^{-\frac{n-p(0)}{p(0)-1}}\Big\}
  \end{equation}
  follow from \eqref{eq:fund-estimate}, $c\geq 1$ being the constant
  implicit in \eqref{eq:fund-estimate}.

  We use the second inclusion in \eqref{eq:incl} to get
  \[
    \int_{\{u>t\}}t^{q(x)}\ud x\leq
    \int_{\big\{t<c\abs{x}^{-\frac{n-p(0)}{p(0)-1}}\big\}}t^{q(x)}\ud x.
  \]
  We use the change of variables
  \[
    \lambda=t^{-\frac{p(0)-1}{n-p(0)}},
  \]
  and obtain that
  \[
    \int_{\big\{t<c\abs{x}^{-\frac{n-p(0)}{p(0)-1}}\big\}}t^{q(x)}\ud x
    \leq c\int_{\{\abs{x}<\lambda\}}\lambda^{-q(x)\frac{n-p(0)}{p(0)-1}}\ud x
    \leq c\int_{\{\abs{x}<\lambda\}}\lambda^{-q(0)\frac{n-p(0)}{p(0)-1}}\ud x
  \]
  where the last estimate follows from the $\log$-H\"older continuity of
  $q$.

  The last integral is finite if \eqref{eq:q-cond} holds.  Starting
  from the first inclusion in \eqref{eq:incl}, we get a similar lower
  bound. Hence $u\in \wL^{\q}(B)$ if and only if
  \eqref{eq:q-cond} holds. Repeating the same argument using
  \eqref{eq:fund-gradient}, we obtain the condition
  \eqref{eq:q-cond-gradient}.
\end{example}

\begin{example}\label{eg:fundamental-sol}
The right hand side of the differential equation in the
previous example is a delta measure, which is not an $L^1$-function.
However, the example can be modified to yield a function in $L^1$. Denote by
$u$ the function from the previous example and define
\[
v_r(x) :=
\begin{cases}
a_r - b_r \, |x| &\text{when } |x|\le r, \\
u(x) &\text{otherwise.}
\end{cases}
\]
The constants $a_r$ and $b_r$ are chosen so that $v_r\in C^1$. Then
$a_r = |\nabla u(r)| \approx r^{-\frac{n-1}{p(0)-1}}$ by the computations
above. A direct calculation shows that
\[
-\dive(\abs{\nabla v_r}^{p(x)-2}\nabla v_r)
= a_r^{p(x)-1} \Big( \frac{n-1}{|x|} + p'(x) \log a_r\Big).
\]
If we suppose that $p$ is Lipschitz continuous, then
\[
\frac{n-1}{|x|} + p'(x) \log a_r \approx \frac{n-1}{|x|}
\]
for small enough $r$ and so the right hand side of this
equation is positive in $B(0,r)$. Furthermore, the right hand side is
in $L^1$ uniformly and $v_r \nearrow u$ as $r\to 0$, so we see that
the conclusions from the previous example hold also for the $L^1$ case.
\end{example}

\paragraph{\textbf{Ackowledgments}}
We would like to thank the referee for useful comments which improved 
the readability of the paper. 

A.\ Almeida was supported in part by {\it FEDER} funds through {\it COMPETE}--Operational Programme Factors of Competitiveness (``Programa Operacional Factores de Competitividade'') and by Portuguese funds through the {\it Center for Research and Development in Mathematics and Applications} (University of Aveiro) and the Portuguese Foundation for Science and Technology (``FCT--Fundação para a Ciência e a Tecnologia''), within project PEst-C/MAT/UI4106/2011 with COMPETE number FCOMP-01-0124-FEDER-022690.

P.\ H\"ast\"o was supported in part by the Academy of Finland and the 
Swedish Cultural Foundation in Finland.



\vspace{0.5cm}

\noindent\small{\textsc{A.\ Almeida}}\\
\small{Center for Research and Development in Mathematics and Applications, Department of Mathematics, University of Aveiro, 3810-193 Aveiro, Portugal}\\
\footnotesize{\texttt{jaralmeida@ua.pt}}\\

\noindent\small{\textsc{P.\ Harjulehto}}\\
\small{Department of Mathematics and Statistics,
FI-20014 University of Turku, Finland}\\
\footnotesize{\texttt{petteri.harjulehto@utu.fi}}\\

\noindent\small{\textsc{P.\ H\"{a}st\"{o}}}\\
\small{Department of Mathematical Sciences,
P.O.\ Box 3000, FI-90014 University of Oulu, Finland}\\
\footnotesize{\texttt{peter.hasto@helsinki.fi}}\\

\noindent\small{\textsc{T.\ Lukkari}}\\
\small{Department of Mathematics and Statistics,
 P.O.Box 35 (MaD), FI-40014 University of Jyv\"{a}skyl\"{a}, Finland}\\
\footnotesize{\texttt{teemu.j.lukkari@jyu.fi}}\\

\end{document}